\documentclass[12pt]{article}%
\usepackage{amssymb}
\usepackage{amsfonts}
\usepackage{graphicx}
\usepackage{amsmath}%
\providecommand{\U}[1]{\protect\rule{.1in}{.1in}}
\newtheorem{theorem}{Teorema}[section]

\newtheorem{lemma}[theorem]{Lemma}

\newtheorem{proposition}[theorem]{Proposition}
\newtheorem{remark}[theorem]{Remark}

\newenvironment{proof}[1][Proof]{\textbf{#1.} }{\ \rule{0.5em}{0.5em}}
\begin{document}

\title{On the Dirichlet problem for the CMC graph equation on multiply connected
domains of a Riemannian manifold}
\author{A. J. Aiolfi, G. S. Nunes, L. O. Sauer, R. B. Soares}
\date{}
\maketitle

\begin{abstract}
We establish existence and uniqueness of compact graphs of cons\nolinebreak
tant mean curvature in $M\times\mathbb{R}$ over bounded multiply connected
domains of\emph{ }$M\times\left\{  0\right\}  $ with boundary lying in two
parallel horizontal slices of $M\times\mathbb{R}.$

\end{abstract}

\section{Introduction}

\qquad Let $M$ be a complete $n-$dimensional Riemannian manifold, $n\geq2$.
Let $\Lambda$ and $\Lambda_{i},$ $i=1,...,m$, be bounded, simply connected
domains of class $C^{2,\alpha}$ of $M$ such that $\overline{\Lambda}%
_{i}\subset\Lambda$ and $\overline{\Lambda}_{i}\cap\overline{\Lambda}%
_{j}=\varnothing$ if $i\neq j$. Consider the multiply connected domain
$\Omega=\Lambda\backslash\left(  \cup_{i=1}^{m}\overline{\Lambda}_{i}\right)
$ and set $\Gamma:=\partial\Lambda$, $\Gamma_{i}:=\partial\Lambda_{i}$. With
this notation\emph{ }%
\[
\partial\Omega=\Gamma\cup\left(  \cup_{i=1}^{m}\Gamma_{i}\right)  \text{.}%
\]
Let $h,H\geq0$ be given. In this paper\ we investigate the Dirichlet problem%
\begin{equation}
\left\{
\begin{array}
[c]{l}%
Q_{H}\left(  u\right)  {\small :=\operatorname{div}}\frac{\nabla u}%
{\sqrt{1+\left\vert \nabla u\right\vert ^{2}}}+nH{\small =0}\text{ in }%
\Omega,\text{ }u\in C^{2,\alpha}\left(  \overline{\Omega}\right) \\
u|_{\Gamma}=0,\text{ }u|_{\Gamma_{i}}=h\text{, }i=1,...,m
\end{array}
\right.  \label{dirH}%
\end{equation}
where $\operatorname{div}$ and $\nabla$ are the divergence and the gradient in
$M$. If $u$ is a solution of (\ref{dirH}) then the graph of $u$ is a compact
constant mean curvature $H$ hypersurface of $M\times\mathbb{R}$ oriented by a
unit normal vector $N$ such that $\left\langle N,d/dt\right\rangle \leq0$ and
whose boundary lies in the slices $M\times\left\{  0\right\}  \cup
M\times\left\{  h\right\}  $.

The Dirichlet problem (\ref{dirH}) was studied in the work \cite{ER} for
$M=\mathbb{R}^{2}$, in \cite{S} when $M=\mathbb{H}^{2}$ and $\partial\Omega$
has only two connected components and in \cite{B} when $M=\mathbb{H}^{n}$ and
$H=0$. In these works existence results are obtained when the height $h$ is
less than or equal to a constant which depends on the mean curvature of
$\partial\Omega$, the distance between the connected component of
$\partial\Omega$, the dimension $n$, and on the diameter of $\Omega$ and $H$
if $H>0$. In \cite{ER} some nonexistence results also are established. In the
case $H=0$ we observe that Theorem 1 of \cite{ARS}, an extension of the
classical Jenkins-Serrin result - Theorem 2 of \cite{JS} -, gives us an
existence result with the upper bound of $h$ depending on $n$, on the mean
curvature and on the injectivity radius of $\partial\Omega$.

The main motivation to study the problem (\ref{dirH}) is that, for $H>0$, we
did not find in the literature, even for $M=\mathbb{R}^{n}$, a result where
the hypersurface $\Gamma$ is not assumed to be mean convex. We explore this
situation when $M$ is a Hadamard manifold (Theorem \ref{H+}). We observe that
in \cite{DR} the authors conclude, for $M=\mathbb{R}^{n}$, the existence and
uniqueness of $H$-graphs for a large class of prescribed boundary data over
$\Gamma$ where $\Gamma$ is not necessarily mean convex but, however,
$\Gamma=\partial\Omega$ with $\Omega$ simply connected.

Relatively to the minimal case, we obtain Theorem \ref{minimal}, whose
estimate on $h$ is more in line with that in Theorem 2.1 of \cite{ER} than
that in Theorem 1 of \cite{ARS}. Despite being difficult to say if our result
improves or not the estimate of $h$ given in \cite{ARS} in general, for some
domains $\Omega$ we got some gain (see Remark \ref{R1}).

An extra motivation to our work is the problem proposed by A. Ros and H.
Rosenberg in Remark 4 of \cite{RR}: Given two convex Jordan curves in distinct
parallel planes of $\mathbb{R}^{3}$, is there a CMC annulus having such curves
as boundary? Besides of the aforementioned works, this situation was also
investigated in \cite{FR}, \cite{AFR} and \cite{AF} (for some characterization
results, see \cite{Sh} and \cite{MW}). The results obtained so far do not give
a complete answer to this problem. Our results give some contribution in the
$M\times\mathbb{R}$ context.

We fix some notations: the mean curvature of $\partial\Omega$ with respect to
the unit normal vector field $\eta$ to $\partial\Omega$ pointing to $\Omega$
will be denoted by $H^{\partial\Omega}$ and
\[
H_{\inf}^{\partial\Omega}:=\inf_{\partial\Omega}H^{\partial\Omega}.
\]

Let $d$ be the Riemannian distance in $M$. Denote by $R_{\Gamma},R_{\Gamma
_{i}}$, the biggest positive numbers such that the exponential maps
\begin{equation}
\exp_{\Gamma}:\Gamma\times\lbrack0,R_{\Gamma})\longrightarrow U_{\Gamma
}:=\left\{  z\in\Omega;\text{ }d(z,\Gamma)<R_{\Gamma}\right\}  \subset
\Omega\label{ugama}%
\end{equation}
and
\begin{equation}
\exp_{\Gamma_{i}}:\Gamma_{i}\times\lbrack0,R_{\Gamma_{i}})\longrightarrow
U_{\Gamma_{i}}:=\left\{  z\in\Lambda\setminus\Lambda_{i};\text{ }%
d(z,\Gamma_{i})<R_{\Gamma_{i}}\right\}  \subset\Lambda\backslash\Lambda_{i}
\label{Ugamai}%
\end{equation}
are diffeomorphisms (here, $\exp\left(  p,s\right)  :=\exp_{p}\left(
s\eta\left(  p\right)  \right)  $) and set
\begin{equation}
R=\min\left\{  R_{\Gamma},R_{\Gamma_{1}},R_{\Gamma_{2}},...,R_{\Gamma_{m}%
}\right\}  . \label{r}%
\end{equation}

\begin{center}
\begin{figure}[tbh]
\centering
\includegraphics[scale=0.4]{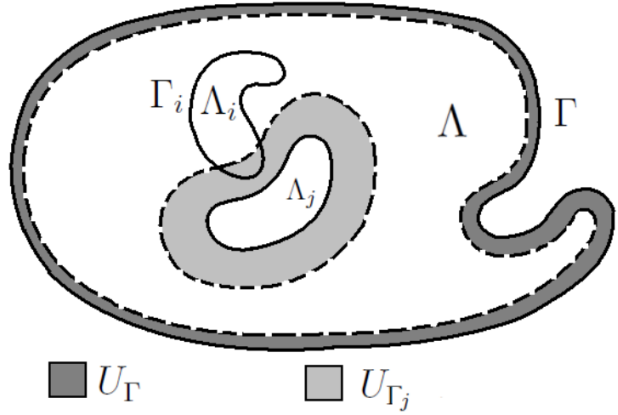}\caption{Examples of $U_{\Gamma
}$ and $U_{\Gamma_{j}}$ domains.}%
\end{figure}
\end{center}

Denote by $\delta$ ($\approx1.8102$) the solution of the equation
\[
x=\cosh\left(  \frac{x}{\sqrt{x^{2}-1}}\right)  \text{, }x>1\text{.}%
\]

We prove:

\begin{theorem}
\label{minimal}Let $M$ be a complete $n-$dimensional Riemannian manifold,
$n\geq2$. If%
\begin{equation}
h<\frac{1}{\delta\left(  n-1\right)  \left\vert H_{\inf}^{\partial\Omega
}\right\vert }\cosh^{-1}\left(  1+\tau\delta\left(  n-1\right)  \left\vert
H_{\inf}^{\partial\Omega}\right\vert \right)  \label{hminimal}%
\end{equation}
where
\[
\tau=\min\left\{  R,\frac{\delta-1}{\delta\left(  n-1\right)  \left\vert
H_{\inf}^{\partial\Omega}\right\vert }\right\}  ,
\]
then the Dirichlet problem (\ref{dirH}) has a unique solution for $H=0$.
\end{theorem}

\bigskip

\begin{theorem}
\label{H+} Assume that $M$ is a Hadamard manifold (complete, simply connected
Riemannian manifold with nonpositive sectional curvature) and that $\Omega$ is
contained in a geodesic ball of radius $\Re$ of $M$. Then there is a positive
constant $C=C\left(  n,H_{\inf}^{\partial\Omega},\Re\right)  $, explicitly
given in (\ref{tc}) such that, given $H\in\left[  0,C\right)  $ we have%
\[
h_{H}:=\frac{\cosh^{-1}(1+\delta\left[  \left(  n-1\right)  \left\vert
H_{\inf}^{\partial\Omega}\right\vert +nH\left(  1+\delta\right)  \right]
\sigma)}{\delta\left[  \left(  n-1\right)  \left\vert H_{\inf}^{\partial
\Omega}\right\vert +nH\left(  1+\delta\right)  \right]  }-\frac{H\Re^{2}%
}{1+\sqrt{1-H^{2}\Re^{2}}}>0,
\]
where
\[
\sigma=\min\left\{  R,\frac{\Re\left(  \delta-1\right)  }{\delta\left[
\Re\left(  n-1\right)  \left\vert H_{\inf}^{\partial\Omega}\right\vert
+n\left(  \delta+1\right)  \right]  }\right\}  ,
\]
and the Dirichlet Problem (\ref{dirH}) has a unique solution if $h\leq h_{H}$.
\end{theorem}

\section{Barriers for the Dirichlet problem (\ref{dirH})}

\qquad We shall use the continuity method from Elliptic PDE theory in the
proof of the main results. Then, we need to construct local barriers
relatively to the Dirichlet problem (\ref{dirH}) (see \cite{GT}, p. 333).

We work with the construction of the local barriers relatively to the points
in $\Gamma$ and in $\Gamma_{i}$ at same time, using $U$ to means both
$U_{\Gamma}$ and $U_{\Gamma_{i}}$ and having in mind that, for $z\in U$,
$d\left(  z\right)  =d\left(  z,\Gamma\right)  $ if $U=U_{\Gamma}$ and
$d\left(  z\right)  =d\left(  z,\Gamma_{i}\right)  $ if $U=U_{\Gamma_{i}}$.

We consider, at $z\in U$, an orthonormal referential frame $\left\{
E_{j}\right\}  $\emph{\ }of\emph{\ }$T_{z}M,$\emph{\ }$j=1,...,n,$
where\emph{\ }%

\begin{equation}
E_{n}:=\nabla d. \label{di}%
\end{equation}

\begin{lemma}
\label{Lcg}Given $\psi\in C^{2}\left(  \left[  0,\infty\right)  \right)  $,
set $s=d\left(  z\right)  $, $z\in\overline{U},$ and consider $w\in
C^{2}\left(  \overline{U}\right)  $ given by
\begin{equation}
w\left(  z\right)  =\psi\left(  s\right)  \label{w}%
\end{equation}
Suppose $H\geq0$. Then $Q_{H}\left(  w\right)  \leq0$ in $U$\emph{ }if%
\begin{equation}
\psi^{\prime\prime}+\left(  \psi^{\prime}+\left[  \psi^{\prime}\right]
^{3}\right)  \Delta d+nH\left(  1+\left[  \psi^{\prime}\right]  ^{2}\right)
^{3/2}\leq0, \label{eg}%
\end{equation}
where $\Delta$ is the Laplacian in $M$.
\end{lemma}

\begin{proof}
We have $Q_{H}\left(  w\right)  =Q_{0}\left(  w\right)  +nH$. After some
calculus we see that $Q_{0}\left(  w\right)  $ is
\begin{equation}
\left(  1+\left\vert \nabla w\right\vert ^{2}\right)  ^{-1/2}%
{\textstyle\sum\limits_{i=1}^{n}}
\left\langle \nabla_{E_{i}}^{\nabla w},E_{i}\right\rangle -\frac{1}{2}\left(
1+\left\vert \nabla w\right\vert ^{2}\right)  ^{-3/2}%
{\textstyle\sum\limits_{i=1}^{n}}
E_{i}\left(  \left\vert \nabla w\right\vert ^{2}\right)  E_{i}\left(
w\right)  . \label{Q0w}%
\end{equation}
Notice that $\nabla w=\psi^{\prime}E_{n}$ and so%
\begin{equation}
E_{i}\left(  \left\vert \nabla w\right\vert ^{2}\right)  =\left\{
\begin{array}
[c]{c}%
2\psi^{\prime}\psi^{\prime\prime}\text{ if }i=n\\
0\text{ if }i\neq n
\end{array}
\right.  , \label{Eiw2}%
\end{equation}%
\begin{equation}
E_{i}\left(  w\right)  =\left\{
\begin{array}
[c]{c}%
\psi^{\prime}\text{ if }i=n\\
0\text{ if }i\neq n
\end{array}
\right.  \label{Eiw}%
\end{equation}
and%
\begin{equation}
\nabla_{E_{i}}^{\nabla w}=\left\{
\begin{array}
[c]{c}%
\psi^{\prime\prime}E_{n}+\psi^{\prime}\nabla_{E_{n}}^{E_{n}}\text{ if }i=n\\
\psi^{\prime}\nabla_{E_{i}}^{E_{n}}\text{ if }i\neq n
\end{array}
\right.  . \label{DgwEi}%
\end{equation}
It follows from (\ref{DgwEi}) that%
\begin{equation}%
{\textstyle\sum\limits_{i=1}^{n}}
\left\langle \nabla_{E_{i}}^{\nabla w},E_{i}\right\rangle =\psi^{\prime
}\operatorname{div}\left(  E_{n}\right)  +\psi^{\prime\prime}=\psi^{\prime
}\Delta d+\psi^{\prime\prime} \label{s1}%
\end{equation}
and from (\ref{Eiw2}) and (\ref{Eiw}) that
\begin{equation}%
{\textstyle\sum\limits_{i=1}^{n}}
E_{i}\left(  \left\vert \nabla w\right\vert ^{2}\right)  E_{i}\left(
w\right)  =2\left[  \psi^{\prime}\right]  ^{2}\psi^{\prime\prime}. \label{s2}%
\end{equation}
Plugging (\ref{s1}) and (\ref{s2}) in (\ref{Q0w}), as $\left\vert \nabla
w\right\vert ^{2}=\left[  \psi^{\prime}\right]  ^{2}$, we obtain
\[
Q_{0}\left(  w\right)  =\left(  1+\left[  \psi^{\prime}\right]  ^{2}\right)
^{-3/2}\left[  \psi^{\prime\prime}+\left(  \psi^{\prime}+\left[  \psi^{\prime
}\right]  ^{3}\right)  \Delta d\right]
\]
and, therefore, $Q_{H}\left(  w\right)  \leq0$ if
\[
\psi^{\prime\prime}+\left(  \psi^{\prime}+\left[  \psi^{\prime}\right]
^{3}\right)  \Delta d+nH\left(  1+\left[  \psi^{\prime}\right]  ^{2}\right)
^{3/2}\leq0.
\]

\end{proof}

\begin{lemma}
\label{LcL}If $H^{\partial\Omega}\geq-c,$ $c>0$, then $\Delta d\leq\left(
n-1\right)  c$ in $U$.
\end{lemma}

\begin{proof}
Since $\partial U$ is compact, there is $0<k<c$ such that
\begin{equation}
Ric_{M}\left(  \eta,\eta\right)  \geq-\left(  n-1\right)  k^{2}, \label{Ric}%
\end{equation}
where $\eta$ is the normal unit vector to $\partial\Omega$ pointing to
$\Omega$. Consider $f:\left[  0,R^{\ast}\right]  \rightarrow\left(
0,+\infty\right)  $ in $C^{2}\left(  \left[  0,R^{\ast}\right]  \right)  $
defined by%
\[
f\left(  t\right)  =k\sinh\left[  \operatorname*{arccoth}\left(  \frac{c}%
{k}\right)  +kt\right]  ,
\]
where $R^{\ast}$ is $R_{\Gamma}$ or $R_{\Gamma_{i}}$, according $U$ is
$U_{\Gamma}$ or $U_{\Gamma_{i}}$. We have
\begin{equation}
\frac{f^{\prime\prime}\left(  t\right)  }{f\left(  t\right)  }=k^{2}%
,t\in\left[  0,R^{\ast}\right]  . \label{f2f}%
\end{equation}
Setting $H_{t}$ the mean curvature of $\partial U_{t}$, where $U_{t}:=\left\{
z\in\overline{U};d(z)\leq t\right\}  $. Since $H^{\partial\Omega}\geq-c$ and
$U_{0}\subset\partial\Omega$, we have%

\begin{equation}
H_{0}\geq H_{\inf}^{\partial\Omega}\geq-c\geq-\frac{f^{\prime}\left(
0\right)  }{f\left(  0\right)  }. \label{f1f}%
\end{equation}
Let $\gamma:\left[  0,R^{\ast}\right]  \longrightarrow U$ be the arc-length
parametrized geodesic such that $\gamma\left(  0\right)  \in U_{0}$ and
$\gamma^{\prime}\left(  t\right)  =\nabla d\left(  \gamma\left(  t\right)
\right)  $. We have from (\ref{Ric}) and (\ref{f2f}) that
\begin{equation}
Ric_{M}\left(  \gamma^{\prime}\left(  t\right)  ,\gamma^{\prime}\left(
t\right)  \right)  \geq-\left(  n-1\right)  \frac{f^{\prime\prime}\left(
t\right)  }{f\left(  t\right)  }, \label{c2}%
\end{equation}
for all $t\in\left[  0,R^{\ast}\right]  $. Since $\nabla d$ is a extension of
$\eta$ to $U$ and (\ref{f1f}), (\ref{c2}) occurs, it follows from Theorem 5.1
of \cite{KR} that%
\[
-H_{t}\left(  \gamma\left(  t\right)  \right)  =\frac{\Delta d\left(
\gamma\left(  t\right)  \right)  }{\left(  n-1\right)  }\leq\frac{f^{\prime
}\left(  t\right)  }{f\left(  t\right)  }.
\]
Then%
\begin{equation}
\Delta d\leq\left(  n-1\right)  \frac{f^{\prime}\left(  t\right)  }{f\left(
t\right)  }=\left(  n-1\right)  k\coth\left[  \operatorname*{arccoth}\left(
\frac{c}{k}\right)  +kt\right]  \leq\left(  n-1\right)  c. \label{estimlapla}%
\end{equation}

\end{proof}

\begin{proposition}
\label{Propos}Suppose $H^{\partial\Omega}\geq-c,$ $c>0$. Given $\lambda
>\alpha>1$ and $H\geq0$ set $\mu:=\left(  n-1\right)  c$ and define
$\psi_{\alpha,\lambda}\left(  s\right)  \equiv\psi\left(  s\right)  $ by%
\begin{equation}
\psi\left(  s\right)  =\frac{1}{\lambda\left[  \mu+nH\left(  1+\lambda\right)
\right]  }\left[  \cosh^{-1}\left(  \alpha+\lambda\left[  \mu+nH\left(
1+\lambda\right)  \right]  s\right)  -\cosh^{-1}\left(  \alpha\right)
\right]  \text{,} \label{carapsi}%
\end{equation}
$s=d\left(  z\right)  $, $z\in U$. Then $w$ as in (\ref{w}) satisfies
$Q_{H}\left(  w\right)  \leq0$ in $U_{\varepsilon}:=\left\{  z\in
U;d(z)<\varepsilon\right\}  $ where
\begin{equation}
\varepsilon=\min\left\{  R,\frac{1}{\mu+nH\left(  1+\lambda\right)  }\left(
\frac{\lambda-\alpha}{\lambda}\right)  \right\}  . \label{eps}%
\end{equation}

\end{proposition}

\begin{proof}
From Lemma \ref{Lcg} and Lemma \ref{LcL} we have $Q_{H}\left(  w\right)
\leq0$ in $U$ if%
\begin{equation}
\psi^{\prime\prime}+\left[  \left(  n-1\right)  c+nH\right]  \left(
\psi^{\prime}\right)  ^{3}+nH\left(  \psi^{\prime}\right)  ^{2}+\left[
\left(  n-1\right)  c+nH\right]  \psi^{\prime}+nH\leq0 \label{o}%
\end{equation}
where $w\left(  z\right)  =\psi\left(  s\right)  $, $s=d\left(  z\right)  $.
Notice that (\ref{o}) can be rewritten as%
\begin{equation}
\psi^{\prime\prime}+B\left(  \psi^{\prime}+\left[  \psi^{\prime}\right]
^{3}\right)  +nH\left(  1+\left[  \psi^{\prime}\right]  ^{2}\right)  \leq0,
\label{or}%
\end{equation}
where $B=\left(  n-1\right)  c+nH=\mu+nH$, with $\mu=\left(  n-1\right)  c$.

We choose%
\[
\psi\left(  s\right)  =a\cosh^{-1}\left(  \alpha+bs\right)  -a\cosh
^{-1}\left(  \alpha\right)  \text{, }\alpha>1
\]
where $a,b$ are positive constants to be determined and $\alpha>1$.

Setting $u\left(  s\right)  =\alpha+bs$, since $u^{\prime}\left(  s\right)
=b$ it follows that%
\[
\psi^{\prime}=\frac{ab}{\left(  u^{2}-1\right)  ^{1/2}},
\]
and
\[
\psi^{\prime\prime}=-\psi^{\prime}\left(  \frac{bu}{u^{2}-1}\right)  .
\]
Then, from (\ref{or}) we see that $Q_{H}\left(  w\right)  \leq0$ if%
\[
\left[  -bu+B\left(  u^{2}-1+a^{2}b^{2}\right)  \right]  \psi^{\prime
}+nH\left(  u^{2}-1+a^{2}b^{2}\right)  \leq0
\]
that is, if%
\[
\left[  -bu+B\left(  u^{2}-1+a^{2}b^{2}\right)  \right]  ab+nH\left(
u^{2}-1+a^{2}b^{2}\right)  \left(  u^{2}-1\right)  ^{1/2}\leq0.
\]
We assume $a,b$ such that $ab=1$. Then the last inequality is true if%
\[
-b+Bu+nHu\left(  u^{2}-1\right)  ^{1/2}\leq0
\]
which is true if%
\[
-b+Bu+nHu^{2}\leq0.
\]
As $u\left(  s\right)  =\alpha+bs$, the last inequality is%
\begin{equation}
nHb^{2}s^{2}+b\left[  2nH\alpha+B\right]  s+B\alpha+nH\alpha^{2}-b\leq0.
\label{orw}%
\end{equation}
Let $\lambda>\alpha$. We assume
\[
b:=\lambda\left[  \mu+nH\left(  1+\lambda\right)  \right]  =B\lambda
+nH\lambda^{2}>B\alpha+nH\alpha^{2}.
\]
We see that (\ref{orw}) is true for $s\in\left[  0,\varepsilon\right]  $,
where $\varepsilon$ is given by (\ref{eps}). This concludes the proof of the proposition.
\end{proof}

\begin{lemma}
\label{cor_pro}Under the hypothesis of Proposition \ref{Propos},
\[
\psi_{\alpha,\lambda}\left(  \varepsilon\right)  <\frac{\cosh^{-1}\left(
\delta\right)  }{\delta\left(  \mu+2nH\right)  }%
\]
where $\delta$ ($\approx1.8102$) is the solution of the equation
$x=\cosh\left(  x\left(  x^{2}-1\right)  ^{-1/2}\right)  $, $x>1$. In
particular,
\begin{equation}
\underset{\alpha\rightarrow1,\lambda\rightarrow\delta}{\lim}\psi
_{\alpha,\lambda}\left(  \varepsilon\right)  =\frac{\cosh^{-1}\left[
1+\delta\left(  \mu+nH\left(  1+\delta\right)  \right)  \rho\right]  }%
{\delta\left(  \mu+nH\left(  1+\delta\right)  \right)  } \label{ro}%
\end{equation}
where
\[
\rho=\min\left\{  R,\frac{1}{\mu+nH\left(  1+\delta\right)  }\left(
\frac{\delta-1}{\delta}\right)  \right\}
\]

\end{lemma}

\begin{proof}
Notice that, since $1<\alpha<\lambda$,
\begin{align*}
\psi_{\alpha,\lambda}\left(  \varepsilon\right)   &  =\frac{\cosh^{-1}\left(
\alpha+\lambda\left(  \mu+nH\left(  1+\lambda\right)  \right)  \varepsilon
\right)  -\cosh^{-1}\left(  \alpha\right)  }{\lambda\left(  \mu+nH\left(
1+\lambda\right)  \right)  }\\
&  <\frac{\cosh^{-1}\left(  \alpha+\lambda\left(  \mu+nH\left(  1+\lambda
\right)  \right)  \varepsilon\right)  }{\lambda\left(  \mu+nH\left(
1+\lambda\right)  \right)  }\\
&  \leq\frac{\cosh^{-1}\left(  \lambda\right)  }{\lambda\left(  \mu+nH\left(
1+\lambda\right)  \right)  }\leq\frac{\cosh^{-1}\left(  \lambda\right)
}{\lambda\left(  \mu+2nH\right)  }.
\end{align*}
Set
\[
f\left(  \lambda\right)  =\frac{\cosh^{-1}\left(  \lambda\right)  }%
{\lambda\left(  \mu+2nH\right)  }\text{, }\lambda>1\text{.}%
\]
We have $f^{\prime}\left(  \lambda\right)  =0$ iff%
\begin{equation}
\lambda=\cosh\left(  \frac{\lambda}{\sqrt{\lambda^{2}-1}}\right)  \text{,
}\lambda>1. \label{eqdelta}%
\end{equation}
The equation (\ref{eqdelta}) has a unique solution $\delta$ ($\approx1.8102$),
and $\delta$ is the maximum (global) point for $f$. Then
\[
\psi_{\alpha,\lambda}\left(  \varepsilon\right)  <\frac{1}{\delta\left(
\mu+2nH\right)  }\cosh^{-1}\left(  \delta\right)  .
\]
The equality (\ref{ro}) follows immediately from definition of $\psi$ and
$\varepsilon$ (observing that both depend on $\alpha$ and $\lambda$).
\end{proof}

\begin{lemma}
\label{height}Let $M$ be a Hadamard manifold.Let $G$ $\subset M\mathbb{\times
R}$ be a compact graph of constant mean curvature $H>0$ over a domain
$\Omega\subset M$ and such that $\partial G=\partial\Omega$ and let $\hbar$ be
the height of $G$. If $\Omega$ is contained in a normal ball in $M$ of radius
$\Re\leq1/H$ then
\begin{equation}
\hbar\leq\frac{H\Re^{2}}{1+\sqrt{1-\left(  H\Re\right)  ^{2}}}.\label{estim}%
\end{equation}

\end{lemma}

\begin{proof}
Let $p$ the center of the normal ball and set $d_{p}\left(  z\right)
=d\left(  z,p\right)  $, $z\in\Omega$. Consider at $z\in\Omega$ an orthonormal
referential frame $\left\{  E_{j}\right\}  $\emph{\ }of\emph{\ }$T_{z}%
M,$\emph{\ }$j=1,...,n$, where\emph{\ }$E_{n}=\operatorname{grad}d_{p}$.

Let $v:=g\circ d_{p}:\Omega\longrightarrow\mathbb{R}$, where%
\[
g\left(  s\right)  =-\sqrt{\frac{1}{H^{2}}-\Re^{2}}+\sqrt{\frac{1}{H}-s^{2}%
}\text{, }s=d_{p}\left(  z\right)  \text{.}%
\]
We have $Q_{H}\left(  v\right)  $ given by
\[
Q_{H}\left(  v\right)  =\left(  1+(g^{\prime})^{2}\right)  ^{\frac{-3}{2}%
}\left[  g^{\prime\prime}+\left(  g^{\prime}+(g^{\prime})^{3}\right)  \Delta
d_{p}+nH\left(  1+(g^{\prime})^{2}\right)  ^{\frac{3}{2}}\right]  .
\]
As the sectional curvature of $M$ is $K_{M}\leq0$, by the Laplacian Comparison
Theorem we have $\Delta d_{p}\geq\Delta d_{E}$, where $d_{E}$ is the Euclidean
distance. Then%
\[
Q_{H}\left(  v\right)  \leq\left(  1+(g^{\prime})^{2}\right)  ^{\frac{-3}{2}%
}\left[  g^{\prime\prime}+\left(  g^{\prime}+(g^{\prime})^{3}\right)  \Delta
d_{E}+nH\left(  1+(g^{\prime})^{2}\right)  ^{\frac{3}{2}}\right]  =0,
\]
where the equality is due that, for $M=\mathbb{R}^{n}$, the graph of $g\left(
s\right)  $ is a spherical cap.

The result follow now of the fact that the maximum height of the graph of $v$
is
\[
\frac{H\Re^{2}}{1+\sqrt{1-\left(  H\Re\right)  ^{2}}}.
\]

\end{proof}

\section{Proof of Theorems}

Proof of Theorem \ref{minimal}

\begin{proof}
Since
\[
h<\frac{1}{\delta\left(  n-1\right)  \left\vert H_{\inf}^{\partial\Omega
}\right\vert }\cosh^{-1}\left(  1+\tau\delta\left(  n-1\right)  \left\vert
H_{\inf}^{\partial\Omega}\right\vert \right)
\]
where
\[
\tau=\min\left\{  R,\frac{\delta-1}{\delta\left(  n-1\right)  \left\vert
H_{\inf}^{\partial\Omega}\right\vert }\right\}  ,
\]
from Lemma \ref{cor_pro} there is $1<\alpha<\delta$, $\alpha$ close enough to
$1$, such that, setting
\begin{equation}
\psi\left(  s\right)  =\frac{1}{\delta\left(  n-1\right)  \left\vert H_{\inf
}^{\partial\Omega}\right\vert }\left[  \cosh^{-1}\left(  \alpha+\delta\left(
n-1\right)  \left\vert H_{\inf}^{\partial\Omega}\right\vert s\right)
-\cosh^{-1}\left(  \alpha\right)  \right]  \text{,} \label{npsi}%
\end{equation}
$s=d\left(  z\right)  $, the function $w\left(  z\right)  =\mathfrak{c}%
+\psi\left(  s\right)  $, $\mathfrak{c}$ constant, satisfies $Q_{0}\left(
w\right)  \leq0$ in $U_{\varepsilon}=\left\{  z\in U;d(z)<\varepsilon\right\}
$ where
\[
\varepsilon=\min\left\{  R,\frac{1}{\left(  n-1\right)  \left\vert H_{\inf
}^{\partial\Omega}\right\vert }\left(  \frac{\delta-\alpha}{\delta}\right)
\right\}  .
\]
Moreover,
\[
h\leq\psi\left(  \varepsilon\right)  <\frac{1}{\delta\left(  n-1\right)
\left\vert H_{\inf}^{\partial\Omega}\right\vert }\cosh^{-1}\left(
1+\tau\delta\left(  n-1\right)  \left\vert H_{\inf}^{\partial\Omega
}\right\vert \right)  .
\]
Notice that
\[
U_{\varepsilon}=\left\{
\begin{array}
[c]{c}%
U_{\Gamma}^{\epsilon}:=\left\{  z\in U_{\Gamma};d(z,\Gamma)<\varepsilon
\right\}  \text{ if }U=U_{\Gamma}\\
U_{\Gamma_{i}}^{\varepsilon}:=\left\{  z\in U_{\Gamma_{i}};d(z,\Gamma
_{i})<\varepsilon\right\}  \text{ if }U=U_{\Gamma_{i}}%
\end{array}
\right.  ,
\]
and then we can consider the supersolutions $w_{\Gamma}\in C^{2}\left(
\overline{U}_{\Gamma}^{\varepsilon}\right)  $ and $w_{\Gamma_{i},t}\in
C^{2}\left(  \overline{U}_{\Gamma_{i}}^{\varepsilon}\right)  $, $t\in\left[
0,1\right]  $, relatively to the operator $Q_{0}$, given by $w_{\Gamma}\left(
z\right)  =\left(  \psi\circ d\right)  \left(  z\right)  $ and $w_{\Gamma
_{i},t}\left(  z\right)  =th+\left(  \psi\circ d\right)  \left(  z\right)  $
respectively, where $\psi\circ d$ is given by (\ref{npsi}). It follows that
$w_{\Gamma}=0$ in $\Gamma=\partial U_{\Gamma}^{\epsilon}\cap\partial\Omega$
and $h\leq w_{\Gamma}\left(  \varepsilon\right)  $ in $\partial U_{\Gamma
}^{\epsilon}\backslash\Gamma$. Moreover, $w_{\Gamma_{i},t}\geq th$ in
$\overline{U}_{\Gamma_{i}}^{\epsilon}$ and, then, in $\overline{U}_{\Gamma
_{i}}^{\varepsilon}\cap\overline{\Omega}$. Let $\varphi\in C^{\infty}\left(
\partial\Omega\right)  $ given by
\begin{equation}
\varphi\left(  p\right)  =\left\{
\begin{array}
[c]{c}%
0\text{ if }p\in\Gamma\\
h\text{ if }p\in\Gamma_{i}\text{, }i=1,...,m
\end{array}
\right.  . \label{nphi}%
\end{equation}
From the Maximum Principle, it follows that\emph{ }$w_{\Gamma}$ and\emph{
}$w_{\Gamma_{i},t}$ are upper local barriers relatively to the boundary data
$t\varphi$ and, for lower local barriers relatively to $t\varphi$, just take
$w_{\Gamma}^{-}=0$ and $w_{\Gamma_{i},t}^{-}=th-\psi\circ d$ in the domains
$\overline{U}_{\Gamma}^{\varepsilon}$ and $\overline{U}_{\Gamma_{i}%
}^{\varepsilon}$ respectively.

Now, set%
\[
V=\{t\in\lbrack0,1];\exists u_{t}\in C^{2,\alpha}\left(  \overline{\Omega
}\right)  \text{ satisfying }Q_{0}\left(  u_{t}\right)  =0,\text{ }%
u_{t}|_{\partial\Omega}=t\varphi\}\text{.}%
\]
We have $V\neq\varnothing$ since $t=0\in V$. Moreover, since $Q_{0}$ is a
uniformly elliptic operator on $C^{2,\alpha}\left(  \overline{\Omega}\right)
$ we can apply the implicit function theorem in Banach spaces to conclude that
$V$ is an open. Now, we apply a standard sequence of arguments to conclude
that $V$ is closed. Let $\left(  t_{n}\right)  \subset V$ a sequence with
$t_{n}\rightarrow t\in\left[  0,1\right]  $. For each $n$, let $u_{t_{n}}\in
C^{2,\alpha}\left(  \overline{\Omega}\right)  $ satisfying $Q_{0}\left(
u_{t_{n}}\right)  =0$, $u_{t_{n}}|_{\partial\Omega}=t_{n}\varphi$. From the
barriers above\emph{ }it follows that the sequence $\left(  u_{t_{n}}\right)
$ has uniformly bounded $C^{0}$ norm. Moreover%
\[
\max_{\partial\Omega}\left\vert \nabla u_{t_{n}}\right\vert \leq\max
_{\partial\Omega}\left\vert \nabla w_{\Gamma}\right\vert <\infty.
\]
It follows of Section $5$ of \cite{DHL} that there is $K>0$ such that
$\max_{\Omega}\left\vert \nabla u_{t_{n}}\right\vert \leq K$ and,
consequently, $\left\vert u_{t_{n}}\right\vert _{1}\leq\overline{K}<\infty$
with the constant $\overline{K}$ independent of $n$. H\"{o}lder estimates and
PDE linear elliptic theory - see \cite{GT} - give us that $\left(  u_{t_{n}%
}\right)  $ is equicontinous in the $C^{2,\beta}$ norm for some $\beta>0$. It
follows that $\left(  u_{t_{n}}\right)  $ contains a subsequence converging
uniformly on the $C^{2}$ norm to a solution $u\in C^{2}\left(  \overline
{\Omega}\right)  $. Regularity theory of linear elliptic PDE (\cite{GT})
implies that $u\in C^{2,\alpha}\left(  \overline{\Omega}\right)  $. Therefore,
$V$ is closed, that is, $V=[0,1]$ and this gives us the existence result. The
uniqueness of the solution is a consequence of the Maximum Principle for the
difference of two solutions.
\end{proof}

\begin{remark}
\label{R1}: Denote by $r$ the biggest positive number such that the normal
exponential map
\begin{equation}
\exp_{\partial\Omega}:\partial\Omega\times\lbrack0,r)\longrightarrow
U_{\partial\Omega}:=\left\{  z\in\Omega;d(z,\partial\Omega)<r\right\}
\subset\overline{\Omega}, \label{Up}%
\end{equation}
is a diffeomorphism. We call $r$ the "injectivity radius of $\partial\Omega$".
Notice that $r\leq R$ and we can have $r<R$. The estimative of $h$ in Theorem
1 of \cite{ARS}, relatively the Dirichlet problem (\ref{dirH}), is $h\leq
B^{-1}\ln\left(  1+B\overline{\varepsilon}\right)  $, where $B=6\left(
1+r^{\prime}\right)  \left(  n-1\right)  \left\vert H_{\inf}^{\partial\Omega
}\right\vert $ and $\overline{\varepsilon}=\min\left\{  r^{\prime},1/\left(
2B\right)  \right\}  $, for some $0<r^{\prime}\leq r$. Then, depending of the
domain $\Omega$, we have some improvement on the estimate of\emph{ }$h$ in
Theorem \ref{minimal} when we compare with that in Theorem 1 of \cite{ARS}.
For example, if $\Omega$ is such that $1/\left(  2B\right)  <r^{\prime}$ and
\[
\frac{\delta-1}{\delta\left(  n-1\right)  \left\vert H_{\inf}^{\partial\Omega
}\right\vert }<R.
\]
Moreover, using $R$ instead of $r$ we are more in line with Theorem 2.1 of
\cite{ER} and Theorem 1.1 of \cite{B}.
\end{remark}

We pass now to the proof of Theorem \ref{H+}

\begin{proof}
Set $\mu=\left(  n-1\right)  \left\vert H_{\inf}^{\partial\Omega}\right\vert
$, and
\begin{equation}
C=\tfrac{2\delta\mu\cosh^{-1}(1+\delta\mu\sigma)}{\left[  \cosh^{-1}%
(1+\delta\left[  \mu+\frac{n}{\Re}\left(  1+\delta\right)  \right]
\sigma)\right]  ^{2}+\mu\delta^{2}\Re\left[  \mu\Re+2n\left(  1+\delta\right)
\right]  +\delta n\left(  1+\delta\right)  \left[  \delta n\left(
1+\delta\right)  -2\cosh^{-1}(1+\delta\mu\sigma)\right]  }, \label{tc}%
\end{equation}
where%
\[
\sigma=\min\left\{  R,\frac{\delta-1}{\delta\left[  \mu+\frac{n}{\Re}\left(
\delta+1\right)  \right]  }\right\}  .
\]
We first notice that
\[
0<C<\frac{1}{\Re}.
\]
In fact, $C>0$ since
\begin{align}
n\delta\left(  1+\delta\right)  -2\cosh^{-1}(1+\delta\mu\sigma)  &  \geq
n\delta\left(  1+\delta\right)  -2\cosh^{-1}\delta\label{m0}\\
&  >5.087-2.3994>0.\nonumber
\end{align}
On the other hand, $C<1/\Re$ since we have (\ref{m0}) and
\[
\frac{2\delta\mu\cosh^{-1}(1+\delta\mu\sigma)}{\mu\delta^{2}\Re\left[  \mu
\Re+2n\left(  1+\delta\right)  \right]  }\leq\frac{2\delta\mu\cosh^{-1}\delta
}{\mu\delta^{2}\Re\left[  2n\left(  1+\delta\right)  \right]  }=\left(
\frac{1}{\Re}\right)  \frac{\cosh^{-1}\delta}{\delta n\left(  1+\delta\right)
}<\frac{1}{\Re}\text{.}%
\]
Let $0\leq H\leq C<1/\Re$. We claim that
\begin{equation}
\frac{\cosh^{-1}(1+\delta\left[  \mu+nH\left(  1+\delta\right)  \right]
\rho)}{\delta\left[  \mu+nH\left(  1+\delta\right)  \right]  }-\frac{H\Re^{2}%
}{1+\sqrt{1-H^{2}\Re^{2}}}>0, \label{le0}%
\end{equation}
where
\[
\rho=\min\left\{  R,\frac{\delta-1}{\delta\left[  \mu+nH\left(  \delta
+1\right)  \right]  }\right\}  .
\]
Indeed, notice that, if $H=0$ then (\ref{le0}) is true. Suppose $H>0$. After
some calculations, we have (\ref{le0}) iff%
\begin{align}
&  \frac{2\delta\mu\cosh^{-1}(1+\delta\left[  \mu+nH\left(  1+\delta\right)
\right]  \rho)}{H}-2\delta^{2}\mu nH\Re^{2}\left(  1+\delta\right)
-\delta^{2}n^{2}H^{2}\Re^{2}\left(  1+\delta\right)  ^{2}\nonumber\\
&  >\left[  \cosh^{-1}(1+\delta\left[  \mu+nH\left(  1+\delta\right)  \right]
\rho)\right]  ^{2}+\nonumber\\
&  +\Re^{2}\delta^{2}\mu^{2}-2\delta n\left(  1+\delta\right)  \cosh
^{-1}(1+\delta\left[  \mu+nH\left(  1+\delta\right)  \right]  \rho).
\label{le1}%
\end{align}
Notice that, since $0<H<1/\Re,$
\begin{align}
\cosh^{-1}(1+\delta\mu\rho)  &  <\cosh^{-1}(1+\delta\left[  \mu+nH\left(
1+\delta\right)  \right]  \rho)\label{le4}\\
&  \leq\cosh^{-1}(1+\delta\left[  \mu+\frac{n}{\Re}\left(  1+\delta\right)
\right]  \rho)\nonumber
\end{align}
and then, from (\ref{le4}), we have that if
\begin{align}
&  \frac{2\delta\mu\cosh^{-1}(1+\delta\mu\rho)}{H}-2\delta^{2}\mu n\Re\left(
1+\delta\right)  -\delta^{2}n^{2}\left(  1+\delta\right)  ^{2}\label{le5}\\
&  >\left[  \cosh^{-1}(1+\delta\left[  \mu+\frac{n}{\Re}\left(  1+\delta
\right)  \right]  \rho)\right]  ^{2}+\Re^{2}\delta^{2}\mu^{2}-2\delta n\left(
1+\delta\right)  \cosh^{-1}(1+\delta\mu\rho)\nonumber
\end{align}
then (\ref{le1}) is true. Notice that since $0<H<1/\Re$ we have $\sigma
\leq\rho$ and, then, (\ref{le5}) occurs if%
\begin{align}
&  \frac{2\delta\mu\cosh^{-1}(1+\delta\mu\sigma)}{H}-2\delta^{2}\mu
n\Re\left(  1+\delta\right)  -\delta^{2}n^{2}\left(  1+\delta\right)
^{2}\label{le6}\\
&  >\left[  \cosh^{-1}(1+\delta\left[  \mu+\frac{n}{\Re}\left(  1+\delta
\right)  \right]  \sigma)\right]  ^{2}+\Re^{2}\delta^{2}\mu^{2}-2\delta
n\left(  1+\delta\right)  \cosh^{-1}(1+\delta\mu\sigma).\nonumber
\end{align}
Notice that (\ref{le6}) is equivalent to $H<C$ and this concludes the proof of
the claim.

As $\sigma\leq\rho$, it follows that%
\begin{equation}
0<h_{H}\leq\frac{\cosh^{-1}(1+\delta\left[  \mu+nH\left(  1+\delta\right)
\right]  \rho)}{\delta\left[  \mu+nH\left(  1+\delta\right)  \right]  }%
-\frac{H\Re^{2}}{1+\sqrt{1-H^{2}\Re^{2}}}, \label{ee}%
\end{equation}
and, as $h\leq h_{H}$, from Lemma \ref{cor_pro} there is $1<\alpha<\delta$,
$\alpha$ close enough to $1$, such that, setting
\begin{equation}
\psi\left(  s\right)  =\frac{\cosh^{-1}(\alpha+\delta\left[  \mu+nH\left(
1+\delta\right)  \right]  s)-\cosh^{-1}\left(  \alpha\right)  }{\delta\left[
\mu+nH\left(  1+\delta\right)  \right]  }\text{,} \label{npsi2}%
\end{equation}
$s=d\left(  z\right)  $, the function $w\left(  z\right)  =\mathfrak{c}%
+\psi\left(  s\right)  $, $\mathfrak{c}$ constant, satisfies $Q_{H}\left(
w\right)  \leq0$ in $U_{\varepsilon}=\left\{  z\in U;d(z)<\varepsilon\right\}
$, where
\[
\varepsilon=\min\left\{  R,\frac{1}{\mu+nH\left(  1+\lambda\right)  }\left(
\frac{\delta-\alpha}{\delta}\right)  \right\}
\]
and, moreover, from (\ref{ee}),
\[
\psi\left(  \varepsilon\right)  \geq h+\frac{H\Re^{2}}{1+\sqrt{1-H^{2}\Re^{2}%
}}.
\]
As
\[
U_{\varepsilon}=\left\{
\begin{array}
[c]{c}%
U_{\Gamma}^{\epsilon}:=\left\{  z\in U_{\Gamma};d(z,\Gamma)<\varepsilon
\right\}  \text{ if }U=U_{\Gamma}\\
U_{\Gamma_{i}}^{\epsilon}:=\left\{  z\in U_{\Gamma_{i}};d(z,\Gamma
_{i})<\varepsilon\right\}  \text{ if }U=U_{\Gamma_{i}}%
\end{array}
\right.  ,
\]
we can consider the supersolutions $w_{\Gamma}\in C^{2}\left(  \overline
{U}_{\Gamma}^{\varepsilon}\right)  $ and $w_{\Gamma_{i},t}\in C^{2}\left(
\overline{U}_{\Gamma_{i}}^{\varepsilon}\right)  $, $t\in\left[  0,1\right]  $,
relatively to the operator $Q_{tH}$, given by $w_{\Gamma}\left(  z\right)
=\left(  \psi\circ d\right)  \left(  z\right)  $ and $w_{\Gamma_{i},t}\left(
z\right)  =th+\left(  \psi\circ d\right)  \left(  z\right)  $ respectively,
where $\psi\circ d$ is given by (\ref{npsi2}). It follows that $w_{\Gamma}=0$
in $\Gamma=\partial U_{\Gamma}^{\epsilon}\cap\partial\Omega$ and
\[
w_{\Gamma}\left(  \varepsilon\right)  \geq h+\frac{H\Re^{2}}{1+\sqrt
{1-H^{2}\Re^{2}}}%
\]
in $\partial U_{\Gamma}^{\epsilon}\backslash\Gamma$. Moreover, $w_{\Gamma
_{i},t}=th$ in $\Gamma_{i}=\partial U_{\Gamma_{i}}^{\epsilon}\cap
\partial\Omega$ and
\[
w_{\Gamma_{i},t}\left(  \varepsilon\right)  \geq h(1+t)+\frac{H\Re^{2}%
}{1+\sqrt{1-H^{2}\Re^{2}}}%
\]
in $\left(  \partial U_{\Gamma_{i}}^{\epsilon}\cap\overline{\Omega}\right)
\backslash\Gamma_{i}$. Let $\varphi$ as in (\ref{nphi}). From Lemma
\ref{height} and from the Maximum Principle, it follows that\emph{ }%
$w_{\Gamma}$ and\emph{ }$w_{\Gamma_{i},t}$ are upper local barriers relatively
to the boundary data $t\varphi\in C^{\infty}\left(  \partial\Omega\right)  $.
From the Theorem \ref{minimal}, since $\sigma\leq\tau$, there is $u\in
C^{2,\alpha}\left(  \overline{\Omega}\right)  $ satisfying $Q_{0}\left(
u\right)  =0$, $u|_{\Gamma}=0$, $u|_{\Gamma_{i}}=h$, $i=1,...,m$. As vertical
translation in $M\times\mathbb{R}$ are isometries, we use vertical
translations of $graf(u)$ to obtain lower barrier relatively to the boundary
data $t\varphi$.

Now, set%
\[
V=\{t\in\lbrack0,1];\exists u_{t}\in C^{2,\alpha}\left(  \overline{\Omega
}\right)  \text{ satisfying }Q_{tH}\left(  u_{t}\right)  =0,\text{ }%
u_{t}|_{\partial\Omega}=t\varphi\}\text{.}%
\]
We have $V\neq\varnothing$ since $t=0\in V$. Moreover, $V$ is open by the
Implicit Function Theorem in Banach spaces. From the barriers above, we obtain
a priori uniform $C^{1}$ estimates for the family of Dirichlet problems
$Q_{0}\left(  u_{t}\right)  =0,$ $u_{t}|_{\partial\Omega}=t\varphi$, which
give us that $V$ is closed (by similar sequence of arguments exposed in the
last paragraph of proof of Theorem \ref{minimal}). The uniqueness of the
solution is a consequence of the Maximum Principle for the difference of two solutions.
\end{proof}

\bigskip

* Ari J. Aiolfi: Dep. de Matem\'{a}tica - Universidade Federal de Santa Maria
, Santa Maria RS/Brazil (ari.aiolfi@mail.ufsm.br);

* Giovanni S. Nunes \& Lisandra O. Sauer: IFM - Universidade Federal de
Pelotas, Pelotas RS/Brazil \newline(giovanni.nunes@ufpel.edu.br, lisandra.sauer@ufpel.edu.br);

* Rodrigo B. Soares: IMEF - Universidade Federal de Rio Grande, Rio Grande
RS/Brazil (rodrigosoares@furg.br)

\end{document}